\documentclass[12pt, a4paper]{amsart}
\usepackage{amscd, amssymb, pb-diagram}
\usepackage{palatino,euler,color}
\usepackage[margin=2.10cm]{geometry}
\usepackage{amsthm}
\usepackage[all]{xy}
\usepackage{tikz}
\usepackage{tile_macros}
\usetikzlibrary{arrows,backgrounds,decorations.markings}

\allowdisplaybreaks

\def\ker{\operatorname{ker}}

\def\Im{\operatorname{Im}}

\def\R{\mathbb{R}}

\def\Z{\mathbb{Z}}

\newtheorem{thm}{Theorem}[section]

\newtheorem{lemma}[thm]{Lemma}
\newtheorem{prop}[thm]{Proposition}

\theoremstyle{definition}

\theoremstyle{remark}

\numberwithin{equation}{section}

\tikzstyle{vertex}=[circle,thick]
\tikzstyle{goto}=[->,shorten >=1pt,>=stealth,semithick]

\pgfpathcircle{\pgfpoint{5cm}{0cm}} {2pt}

\begin{document}

\date{\today}
\title[Cohomology of the Pinwheel Tiling]{Cohomology of the Pinwheel Tiling}

\author[Dirk Frettl\"{o}h]{Dirk Frettl\"{o}h}
\address{Dirk Frettl\"{o}h, Faculty of Technology, Universit\"{a}t Bielefeld, Postfach 100131,33501 Bielefeld, Germany}
\email{dirk.frettloeh@math.uni-bielefeld.de}
\author[Benjamin Whitehead]{Benjamin Whitehead}
\address{Benjamin Whitehead, School of Mathematics and Applied Statistics, University of Wollongong, NSW 2522, Australia}
\email{bw219@uowmail.edu.au}
\author[Michael F. Whittaker]{Michael F. Whittaker}
\address{Michael F. Whittaker, School of Mathematics and Applied Statistics, University of Wollongong, NSW 2522, Australia}
\email{mfwhittaker@gmail.com}

\thanks{This research was supported by the German Research Council (DFG) within the Collaborative Research Centre 701, the Australian Mathematical Sciences Institute, and the Australian Research Council.}

\begin{abstract}
We provide a computation of the cohomology of the Pinwheel tiling using the Anderson-Putnam complex. A border forcing version of the Pinwheel tiling is constructed that allows an explicit construction of the complex for the quotient of the continuous hull by the circle. The final result is given using a spectral sequence argument of Barge, Diamond, Hunton, and Sadun. 
\end{abstract}

\maketitle

\section{Introduction}

A tiling of the plane is a collection of subsets of $\mathbb{R}^{2}$, called tiles, for which any intersection of the interiors of two distinct tiles is empty and whose union is all of $\mathbb{R}^2$. A tiling said to be {\em aperiodic} if it lacks translational periodicity. The most common method of producing aperiodic tilings is to use a substitution rule; a method for breaking each tile into smaller pieces, each of which is a scaled down copy of one of the original tiles, and then expanding so that each tile is congruent to one of the original tiles.

One of the most enigmatic substitution tilings is the Conway-Radin Pinwheel tiling, described in the seminal paper \cite{Rad}. The Pinwheel tiling is composed of two tile types, a $1$-$2$-$\sqrt{5}$ right triangle and its reflection. The substitution rule $\omega$ inflates each tile by a factor of $\sqrt{5}$ and decomposes the inflated triangle into 5 isometric copies of the original triangles as follows:
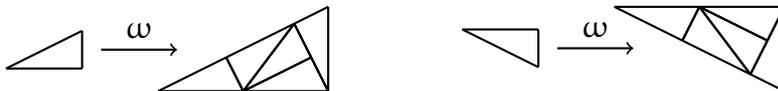
\begin{center} 
\begin{tikzpicture}[scale=1.0]
\begin{scope}[thick,xshift=0cm,yshift=0cm,rotate=0]
\begin{scope}[xshift=0.0cm,yshift=0.3cm,rotate=0]
\PLeft{0}{0}{0}
\draw[->] (1.25,0.25) -- (2.25,0.25);
\node at (1.75,0.5) {$\omega$};
\end{scope}
\begin{scope}[xshift=2.0cm,yshift=0cm,rotate=0]
\PRight{0}{0}{26.565} 
\PRight{1.118}{0}{26.565}
\PLeft{1.118}{0}{26.565}
\PLeftx{1.118}{0}{206.565}
\PRight{2.236}{0}{116.565}
\end{scope}
\end{scope}
\begin{scope}[thick,xshift=6cm,yshift=1.118cm,rotate=0,yscale=-1]
\begin{scope}[xshift=0.0cm,yshift=0.3cm,rotate=0]
\PLeft{0}{0}{0}
\draw[->] (1.25,0.25) -- (2.25,0.25);
\node at (1.75,0.0) {$\omega$};
\end{scope}
\begin{scope}[xshift=2.0cm,yshift=0cm,rotate=0]
\PRight{0}{0}{26.565} 
\PRight{1.118}{0}{26.565}
\PLeft{1.118}{0}{26.565}
\PLeftx{1.118}{0}{206.565}                                                                                                                                                                                                   
\PRight{2.236}{0}{116.565}
\end{scope}
\end{scope}
\end{tikzpicture}
\end{center}
The reason the Pinwheel tiling has been so difficult to analyse is twofold; tiles in the Pinwheel appear in an infinite number of distinct orientations and the Pinwheel tiling does not force its border. A substitution tiling is said to force its border if, for every proto-tile, there is a finite number of substitutions after which the tiles meeting the border of the collection of substituted tiles are known. The notion of border forcing was invented by Kellendonk in \cite[p.24]{Kel1}. Forcing the border has turned out to be essential for computing the cohomology of tilings \cite{AP,BDHS,Lorenzo.book}. To see that tiles in the Pinwheel tiling appear in an infinite number of distinct orientations we observe that after two iterations of the Pinwheel substitution there are two tiles that are rotated against each other by the angle $\arctan(\frac12)$, which is an irrational fraction of $2\pi$ so an induction argument shows that orientations of tiles are dense in the circle.

There is a compact topological space associated with the Pinwheel tiling called the continuous hull and denoted $\Omega$. The continuous hull is locally homeomorphic to the product of a Cantor set, a circle, and a disc \cite{ORS,BG,Sad,Lorenzo.book, Whi}. The substitution rule $\omega$ is a homeomorphism on the continuous hull providing a method to reconstruct the continuous hull using techniques invented by Anderson and Putnam in \cite{AP}. Computing the cohomology groups of the continuous hull has been the focus of a significant amount of research since its introduction in \cite{AP}.

 In \cite{BDHS}, Barge, Diamond, Hunton, and Sadun give the first computation of the cohomology of the Pinwheel tiling. Their computation uses Barge-Diamond collaring and a higher dimensional analogue of the Barge-Diamond complex \cite{BD} to compute the cohomology of the quotient of the continuous hull by the circle. In order to obtain the cohomology of the continuous hull the obvious approach is to use the Gysin sequence to realise the continuous hull as a fibre bundle over the quotient. Unfortunately the map from the continuous hull to the quotient of the continuous hull by the circle is not an honest fibration, there are six smooth singular fibres. In the quotient space these singular fibres correspond to cone singularities, specifically at the points with $180^\circ$ rotational symmetry. To overcome this problem the authors go back to the spectral sequence used to derive the Gysin sequence and adjust for the singular fibres. Torsion in the quotient of the continuous hull by the circle make the spectral sequence argument extremely subtle.

In this note we compute the cohomology of the continuous hull of the Pinwheel tiling using Anderson and Putnam's original approach. The first step is to produce a border forcing version of the Pinwheel tiling, which we call the BFPinwheel. The Anderson-Putnam complex of the BFPinwheel leads to the cohomology of an approximant. A direct limit computation gives the cohomology of the quotient of the continuous hull of the BFPinwheel by the circle. We are able to explicitly present the generators of the cohomology groups and attempt to provide all the details of the computation. To complete the calculation we identify the cone singularities in the BFPinwheel and apply the spectral sequence arguments in \cite{BDHS}.

\subsection{Acknowledgements}
We would like to thank Michael Baake, Lorenzo Sadun, and Aidan Sims for several helpful suggestions. We are especially grateful to Lorenzo for expanding on the constructions in \cite{BDHS} and providing insightful feedback on an early version of the paper.

\section{A border forcing version of the Pinwheel tiling}

In this section we produce a version of the Pinwheel tiling that forces it border. The first step is to pass to a kite-rectangle version of the Pinwheel tiling and use this version to produce the BFPinwheel.

In \cite{BFG1}, Baake, Frettl\"{o}h, and Grimm observed that given a Pinwheel tiling one can remove every hypotenuse to get a kite-rectangle version of the Pinwheel tiling, which we will call the KRPinwheel. They showed that the KRPinwheel tiling is mutually locally derivable to the original Pinwheel tiling, meaning that there is a homeomorphism between the two continuous hulls and the homeomorphism can be determined locally in every tiling. The following figure shows the local bijection from the Pinwheel to the KRPinwheel. The dots emphasise the decomposition although the dots can be completely determined by local patches.
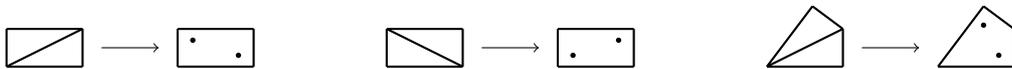
\begin{center} 
\begin{tikzpicture}[scale=1.0]
\begin{scope}[xshift=0cm,yshift=0cm]
\begin{scope}[thick,xshift=0.0cm,yshift=0.0cm,rotate=0]
\draw (0,0) -- (1,0);
\draw (1,0) -- (1,0.5);
\draw (0,0) -- (1,0.5);
\end{scope}
\begin{scope}[thick,xshift=1.0cm,yshift=0.5cm,rotate=180]
\draw (0,0) -- (1,0);
\draw (1,0) -- (1,0.5);
\end{scope}
\begin{scope}
\draw[->] (1.25,0.25) -- (2.0,0.25);
\end{scope}
\begin{scope}[xshift=2.25cm,yshift=0cm]
\Left{0}{0}{0}
\end{scope}
\end{scope}
\begin{scope}[xshift=5cm,yshift=0cm]
\begin{scope}[thick,xshift=0.0cm,yshift=0.0cm,rotate=0]
\draw (0,0) -- (1,0);
\draw (1,0) -- (1,0.5);
\draw (0,0.5) -- (1,0);
\end{scope}
\begin{scope}[thick,xshift=1.0cm,yshift=0.5cm,rotate=180]
\draw (0,0) -- (1,0);
\draw (1,0) -- (1,0.5);
\end{scope}
\begin{scope}
\draw[->] (1.25,0.25) -- (2.0,0.25);
\end{scope}
\begin{scope}[xshift=2.25cm,yshift=0cm]
\Right{0}{0}{0}
\end{scope}
\end{scope}
\begin{scope}[xshift=10cm,yshift=0cm]
\begin{scope}[thick,xshift=0.0cm,yshift=0.0cm,rotate=0]
\draw (0,0) -- (1,0);
\draw (1,0) -- (1,0.5);
\draw (0,0) -- (1,0.5);
\end{scope}
\begin{scope}[thick,xshift=0.0cm,yshift=0.0cm,rotate=53.13]
\draw (0,0) -- (1,0);
\draw (1,0) -- (1,-0.5);
\end{scope}
\begin{scope}
\draw[->] (1.25,0.25) -- (2.0,0.25);
\end{scope}
\begin{scope}[xshift=2.25cm,yshift=0cm]
\KLeft{0}{0}{0}
\end{scope}
\end{scope}
\end{tikzpicture}
\end{center}
The cost of going to the KRPinwheel tiling is that we now need two of the original Pinwheel substitutions in order to achieve the KRPinwheel substitution, which we also denote by $\omega$.
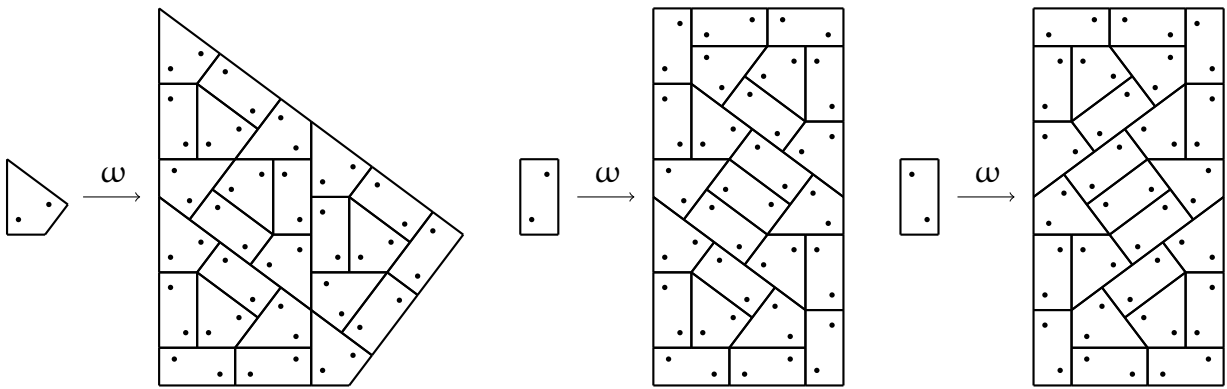
\begin{center} 
\begin{tikzpicture}[scale=1.0]
\begin{scope}[xshift=6.5cm,yshift=0cm,rotate=0]
\begin{scope}[xshift=-0.0cm,yshift=0cm,rotate=0]
\Left{-2.25}{2.5}{-90}
\draw[->] (-1.5,2) -- (-0.75,2);
\node at (-1.1,2.3) {$\omega$};
\end{scope}
\begin{scope}[xshift=0cm,yshift=0cm]
\Left{0.0}{3.5}{90} \Right{0.0}{4.0}{0} \Left{1.0}{4.0}{0}
\Right{0.0}{2.5}{90} \KLeft{0.0}{3.5}{-90} \KLeft{1.0}{4.0}{180} \KLeft{1.5}{3.0}{90}\Right{2.0}{3.0}{90}
\KLeftx{-0.5}{2.0}{180} \Leftx{0.7}{3.6}{-36.87} \KLeft{2.0}{2.0}{90}
\Left{1.0}{1.5}{143.13} \Rightx{1.0}{1.5}{143.13} \Leftx{1.3}{1.9}{143.13} \KLeft{1.0}{1.5}{0}
\KLeft{-0.5}{2.0}{-90}  \Leftx{0.8}{0.4}{143.13} \KLeft{1.5}{0.5}{90} \Right{2.0}{0.5}{90}
\Right{0.0}{0.0}{90} \KLeft{0.0}{1.0}{-90} \KLeft{0.5}{0.0}{0}
\Left{-0.5}{-0.5}{0} \Right{0.5}{-0.5}{0} \Left{2.0}{-0.5}{90}
\end{scope}
\end{scope}
\begin{scope}[xshift=13cm,yshift=0cm,rotate=0]
\begin{scope}[xshift=-1.5cm,yshift=0cm,rotate=0]
\Right{-2.25}{2.5}{-90}
\draw[->] (-1.5,2) -- (-0.75,2);
\node at (-1.1,2.3) {$\omega$};
\end{scope}
\begin{scope}[xshift=0cm,yshift=0cm, xscale=-1]
\Left{0.0}{3.5}{90} \Right{0.0}{4.0}{0} \Left{1.0}{4.0}{0}
\Right{0.0}{2.5}{90} \KLeft{0.0}{3.5}{-90} \KLeft{1.0}{4.0}{180} \KLeft{1.5}{3.0}{90}\Right{2.0}{3.0}{90}
\KLeftx{-0.5}{2.0}{180} \Leftx{0.7}{3.6}{-36.87} \KLeft{2.0}{2.0}{90}
\Left{1.0}{1.5}{143.13} \Rightx{1.0}{1.5}{143.13} \Leftx{1.3}{1.9}{143.13} \KLeft{1.0}{1.5}{0}
\KLeft{-0.5}{2.0}{-90}  \Leftx{0.8}{0.4}{143.13} \KLeft{1.5}{0.5}{90} \Right{2.0}{0.5}{90}
\Right{0.0}{0.0}{90} \KLeft{0.0}{1.0}{-90} \KLeft{0.5}{0.0}{0}
\Left{-0.5}{-0.5}{0} \Right{0.5}{-0.5}{0} \Left{2.0}{-0.5}{90}
\end{scope}
\end{scope}
\begin{scope}[xshift=0cm,yshift=0cm,rotate=0]
\begin{scope}[xshift=-0.0cm,yshift=0cm,rotate=0]
\KLeft{-2.5}{2.5}{-90}
\draw[->] (-1.5,2) -- (-0.75,2);
\node at (-1.1,2.3) {$\omega$};
\end{scope}
\begin{scope}[xshift=-0.5cm,yshift=4.5cm,rotate=180]
\begin{scope}[xshift=-2cm,yshift=0.5cm]
\Lefth{0.0}{3.5}{90} \Right{0.0}{4.0}{0} \Left{1.0}{4.0}{0}
\KLefth{0.0}{3.5}{-90} \KLeft{1.0}{4.0}{180} \KLeft{1.5}{3.0}{90}\Right{2.0}{3.0}{90}
\Leftx{0.7}{3.6}{-36.87} \KLeft{2.0}{2.0}{90}
\Rightxh{1.0}{1.5}{143.13} \Leftx{1.3}{1.9}{143.13} \KLeft{1.0}{1.5}{0}
\Lefth{1.5}{0.5}{90} \Right{2.0}{0.5}{90}
\Lefth{2.0}{-0.5}{90}
\end{scope}
\begin{scope}[thick,rotate=53.13, xscale=-1]
\begin{scope}[xshift=-2cm,yshift=0.5cm]
\Lefth{0.0}{3.5}{90} \Right{0.0}{4.0}{0} \Left{1.0}{4.0}{0}
\KLefth{0.0}{3.5}{-90} \KLeft{1.0}{4.0}{180} \KLeft{1.5}{3.0}{90}\Right{2.0}{3.0}{90}
\Leftx{0.7}{3.6}{-36.87} \KLeft{2.0}{2.0}{90}
\Rightxh{1.0}{1.5}{143.13} \Leftx{1.3}{1.9}{143.13} \KLeft{1.0}{1.5}{0}
\Lefth{1.5}{0.5}{90} \Right{2.0}{0.5}{90}
\Lefth{2.0}{-0.5}{90}
\end{scope}
\end{scope}
\end{scope}
\end{scope}
\end{tikzpicture}
\end{center}

Using the KRPinwheel we construct a border forcing version of the Pinwheel tiling which we achieve by producing a set of collared tiles. We will call this new version the BFPinwheel tiling. A collared tile is a tile labelled by the pattern of its nearest neighbours; that is, two tiles have distinct labels if they are different tiles or if there are two different tiles sharing any edge or vertex with the tile. The BFPinwheel is constructed by considering all possible collared tiles in the KRPinwheel. Labelling the collared tiles gives 83 distinct tiles, 31 kites and 26 of each oriented rectangle. The collared tiles appear in Appendix \ref{collared}. It is important to note that each tile is still an isometric copy of one the KRPinwheel tiles and the labelling merely specifies the collar.

\section{The Anderson-Putnam complex of the Pinwheel tiling}\label{cohomologyofgamma}

In this section we compute the cohomology of the quotient space $\Omega_0 = \Omega/S^1$ for the BFPinwheel. Computing cohomology for tilings with infinite rotational symmetry is the topic of \cite[Chapter 4]{Lorenzo.book} and we direct the reader there for further information. For the Pinwheel tiling, the quotient $\Omega_0$ has been studied by several authors \cite{BDHS, Lorenzo.book, Whi} and is equivalent to considering the subset of tilings in the continuous hull with the tile containing the origin of $\mathbb{R}^2$ in a fixed standard orientation. The same holds true for the BFPinwheel and we declare the orientations of the 83 tiles in Appendix \ref{app_tiles} to be the standard orientation.

Starting with the 83 tiles, a CW complex $\Gamma/S^1$ is defined by identifying edges and vertices in two tiles whenever they are common to both tiles in any patch of tiles appearing in a KRPinwheel tiling, and then arranging the tiles in standard orientation. It will be useful to label the 83 tiles in terms of their tile type: the 31 kite tiles are labelled $K1, \ldots, K31$, 26 left rectangle tiles by $L1, \ldots,L26$, and 26 right rectangle tiles by $R1, \ldots,R26$. After identification of vertices and edges we obtain 138 edges, labelled $e_1, \ldots, e_{138}$, and 73 vertices, labelled $v_1, \ldots, v_{73}$. The CW complex $\Gamma/S^1$ appears in Appendix \ref{app_tiles}.

To use Anderson and Putnam's machinery we have from \cite[\S 4.4]{Lorenzo.book} that
\[
\Omega_0 = \Omega / S_1 = \underset{\leftarrow}{\lim} (\Gamma/S^1 \leftarrow \Gamma/S^1 \leftarrow \Gamma/S^1 \cdots )
\]
with the map induced by the substitution on $\Gamma/S^1$. Since cohomology is contravariant, inverse limits of space turns into direct limits of groups. Thus, we will compute the cohomology of $\Omega_0$ via
\begin{equation}\label{inv_limit}
\check{H}^i(\Omega_0)=\underset{\rightarrow}{\lim}(H^i(\Gamma/S^1),A_i^*)
\end{equation}
where $A_i^*$ denotes the map induced by the substitution on $H^i(\Gamma)$.

To compute the integer cohomology of $\Gamma/S^1$ we begin by defining the cochain complex. Let $C^0$, $C^1$ and $C^2$ denote the functions from the vertices, edges, and tiles into the group of integers, respectively. The boundary maps $\delta_0: C^0 \rightarrow C^1$ and $\delta_1: C^1 \rightarrow C^2$ are defined using the dual of the homology boundary maps. More precisely, we define a matrix for the homology boundary maps that gives rise to the boundary maps in cohomology by taking the transpose of the matrix\footnote{The substitution and boundary matrices are available upon request as Maple worksheets and in PDF form.}. The cochain complex is given in the diagram below which is used to compute the cohomology groups of $\Gamma/S^1$.
\begin{align*}
C^0 & \overset{\delta_0}{\longrightarrow} C^1 \overset{\delta_1}{\longrightarrow} C^2 \overset{\delta_2}{\longrightarrow} 0 \\
\mathbb{Z}^{73} & \overset{\delta_0}{\longrightarrow} \mathbb{Z}^{138} \overset{\delta_1}{\longrightarrow} \mathbb{Z}^{83} \overset{\delta_2}{\longrightarrow} 0.
\end{align*}

From general theory we have that $H^0(\Gamma/S^1) = \ker(\delta_0)$. The kernel is generated by the function that assigns the value one to every vertex and we obtain $H^0(\Gamma/S^1) \cong \Z$.

The computation of the remainder of the cohomology groups is more complicated. We have $H^{1}(\Gamma/S^1)=\ker(\delta_1)/\textrm{Im}(\delta_0)$. For $1 \leq i \leq 138$, by abuse of notation let $e_i$ denote the function in $C^1$ which assigns $1$ to the edge $e_i$ and $0$ to all other edges; that is, the pointmass function in $C^1$ associated to the edge $e_i$. Our computations show that the group $H^1(\Gamma/S^1)$ is generated by the class represented by the function
\begin{align}
\notag
f_1:=&-e_1+e_3-e_7+e_8-e_{10}+e_{14}+e_{17}-e_{18}+e_{19}-e_{26}+e_{30}+e_{32}+e_{35} \\
\label{function_h1}
&+e_{36}-e_{38}+e_{42}-e_{45}-e_{49}+e_{51}+e_{58}+e_{61}+e_{63}+e_{65}+e_{70}-e_{74}-e_{78}\\
\notag
&+e_{79}-e_{80}-e_{81}+e_{83}-e_{84}-e_{85}-e_{90}+e_{95}-e_{99}+e_{101}+e_{106}-e_{109}+e_{125}.
\end{align}
Thus, we have $H^{1}(\Gamma/S^1) \cong \mathbb{Z}$.

The computation of $H^{2}(\Gamma/S^1)=\ker(\delta_2)/\Im(\delta_1)$ is even more involved. Let $t_1, \ldots,t_{83}$ denote the pointmass functions in $C^2$ associated to the corresponding tile in $K1, \ldots, K31$, $L1, \ldots,L26$, $R1, \ldots,R26$ in order. For example, $t_{12}$ assigns the value $1$ to tile $K12$ and $0$ to all other tiles. Similarly, $t_{56}$ is the pointmass function associated with tile $L25$. Our computations show that the group $H^{2}(\Gamma/S^1)$ is generated by the classes containing the 19 functions $g_j \in C^2$ described below in \eqref{generator_h2}. For $j=1, \ldots, 19$, denote the entries of the vectors in Appendix \ref{matrices} by $y_{ij}$ for $i = 1, \ldots, 83$ and let 
\begin{equation}\label{generator_h2}
g_j = \sum_{i=1}^{83} y_{ij} t_i \qquad j\in\{1, \ldots, 19\}.
\end{equation}
The classes of $H^2(\Gamma/S^1)$ are represented by $g_j + \Im(\delta_1)$.

We obtain $H^2(\Gamma/S^1) \cong \mathbb{Z}^{18} \oplus \mathbb{Z}_2$ where the generator of $\Z_2$ is represented by the function $g_{19}$. More specifically, the generator of $\mathbb{Z}_2$ is $t_{59}-2 t_{64}+\Im(\delta_1)$ where $t_{59}$ and $t_{64}$ are the pointmass functions associated with tiles $R2$ and $R7$ respectively. 

We have proven:

\begin{prop}\label{gamma/S}
Let $\Gamma/S^1$ be the Anderson-Putnam complex of the KRPinwheel tiling described in Appendix \ref{app_tiles}. The integer cohomology groups are given by
\begin{align*}
H^{0}(\Gamma/S^1)&\cong \mathbb{Z}\\
H^{1}(\Gamma/S^1)& \cong \mathbb{Z}\\
H^{2}(\Gamma/S^1)& \cong\mathbb{Z}^{18}\oplus\mathbb{Z}_2.
\end{align*}
\end{prop}

The next step is to compute the \^{C}ech cohomology of $\Omega_0$ using the inverse limit construction in Equation \eqref{inv_limit}. We state the main result of the section.

\begin{thm}\label{Cech}
Let $\Omega_0$ be the quotient of the continuous hull $\Omega/S^1$ for the KRPinwheel tiling. The \^{C}ech cohomology groups of $\Omega_0$ are given by
\begin{align*}
\check{H}^{0}(\Omega_0)&\cong\mathbb{Z}\\
\check{H}^{1}(\Omega_0)&\cong\mathbb{Z}\\
\check{H}^2(\Omega_0)&\cong \mathbb{Z}[\frac{1}{25}]\oplus\mathbb{Z}[\frac{1}{3}]^2\oplus \mathbb{Z}^5 \oplus \mathbb{Z}_2.
\end{align*}
\end{thm}

We will prove Theorem \ref{Cech} using a sequence of lemmas.

\begin{lemma}
Let $\Omega_0$ be the quotient of the continuous hull $\Omega/S^1$ for the KRPinwheel tiling, then $\check{H}^{0}(\Omega_0) \cong\mathbb{Z}$.
\end{lemma}

\begin{proof}
The induced matrix $A_0^*$ is the identity on the generator of $H^0(\Gamma/S^1)$. Therefore,
\[
\check{H}^{0}(\Omega_0)=\underset{\rightarrow}{\lim}(H^0(\Gamma/S^1),A_0^*)\cong\underset{\rightarrow}{\lim}(\Z,id)\cong\mathbb{Z}. \qedhere
\]
\end{proof}

\begin{lemma}
Let $\Omega_0$ be the quotient of the continuous hull $\Omega/S^1$ for the KRPinwheel tiling, then $\check{H}^1(\Omega_0) \cong\mathbb{Z}$.
\end{lemma}

\begin{proof}
The induced matrix $A_1^*$ is the identity map on the generator $f_1+\Im \delta_0$ of $H^1(\Gamma/S^1)$, where $f_1$ is described in \eqref{function_h1}. Therefore,
\[
\check{H}^{1}(\Omega_0)=\underset{\rightarrow}{\lim}(H^1(\Gamma/S^1),A_1^*)\cong\underset{\rightarrow}{\lim}(\Z,id)\cong\mathbb{Z}.\qedhere
\]
\end{proof}

The group $\check{H}^2(\Omega_0)$ is the most difficult to compute.

\begin{lemma}
Let $\Omega_0$ be the quotient of the continuous hull $\Omega/S^1$ for the KRPinwheel tiling, then $\check{H}^{2}(\Omega_0) \cong\mathbb{Z}[\frac{1}{25}]\oplus\mathbb{Z}[\frac{1}{3}]^2\oplus \mathbb{Z}^5 \oplus \mathbb{Z}_2$.
\end{lemma}

\begin{proof}
We have $H^2(\Gamma/S^1)\cong \mathbb{Z}^{18} \oplus \mathbb{Z}_2$. We write $A_2^*: H^2(\Gamma/S^1) \to H^2(\Gamma/S^1)$ as a matrix with respect to the 19 generators $g_i$ of $H^2(\Gamma/S^1)$ introduced in Equation \eqref{generator_h2}:
\[
A_2^*=
\left(\begin{smallmatrix} 25&0&0&0&0&0&0&0&0&0&0&0&0
&0&0&0&0&0&0\\ \noalign{\medskip}0&3&0&0&0&0&0&0&0&0&0&0&0&0&0&0&0&0&0
\\ \noalign{\medskip}0&0&3&0&0&0&0&0&0&0&0&0&0&0&0&0&0&0&0
\\ \noalign{\medskip}0&0&0&1&0&0&0&0&0&0&0&0&0&0&0&0&0&0&0
\\ \noalign{\medskip}0&0&0&0&1&0&0&0&0&0&0&0&0&0&0&0&0&0&0
\\ \noalign{\medskip}0&0&0&0&0&1&0&0&0&0&0&0&0&0&0&0&0&0&0
\\ \noalign{\medskip}0&0&0&0&0&0&0&0&0&0&0&0&0&0&0&0&0&0&0
\\ \noalign{\medskip}0&0&0&0&0&0&0&0&0&0&0&0&0&0&0&0&0&0&0
\\ \noalign{\medskip}0&0&0&0&0&0&0&0&0&0&0&0&0&0&0&0&0&0&0
\\ \noalign{\medskip}0&0&0&0&0&0&0&0&0&0&0&0&0&0&0&0&0&0&0
\\ \noalign{\medskip}0&0&0&0&0&0&0&0&0&0&0&0&0&0&0&0&0&0&0
\\ \noalign{\medskip}0&0&0&0&0&0&0&0&0&0&0&0&0&0&0&0&0&0&0
\\ \noalign{\medskip}0&0&0&0&0&0&0&0&0&0&0&0&0&0&0&0&0&0&0
\\ \noalign{\medskip}612&25&-150&0&0&0&110&0&-110&0&0&-110&0&1&0&0&0&0
&0\\ \noalign{\medskip}102&50&250&0&0&0&110&0&-110&0&0&-110&0&1&0&0&0&0
&0\\ \noalign{\medskip}567&100&1050&275&-275&825&110&0&-110&0&0&-110&0
&1&0&0&0&0&0\\ \noalign{\medskip}382&-50&300&0&550&-2200&110&0&-110&0&0
&-110&0&1&0&0&0&0&0\\ \noalign{\medskip}216&25&-150&0&0&0&-220&0&220&0
&0&220&0&0&0&0&0&1&0\\ \noalign{\medskip}0&0&0&0&0&0&0&0&0&0&0&0&0&0&0
&0&0&0&1\end{smallmatrix}\right)
\]
where the last coordinate is the generator of $\mathbb{Z}_2$. We compute the direct limit
\[
\check H^2(\Omega_0) =\displaystyle\lim_{\longrightarrow}\left( H^2(\Gamma), A_2^*\right).
\]
The following vectors generate the image $A_2^* H^2(\Gamma/S^1)$:
\[
\left \{ \left(\begin{smallmatrix} 25\\ \noalign{\smallskip} 0\\ \noalign{\smallskip} 0\\ \noalign{\smallskip} 0\\ \noalign{\smallskip} 0\\ \noalign{\smallskip} 0\\ \noalign{\smallskip} 0\\ \noalign{\smallskip} 0\\ \noalign{\smallskip} 0\\ \noalign{\smallskip} 0\\ \noalign{\smallskip} 0\\ \noalign{\smallskip} 0\\ \noalign{\smallskip} 0\\ \noalign{\smallskip} 0\\ \noalign{\smallskip} -510\\ \noalign{\smallskip} -45\\ \noalign{\smallskip} -230\\ \noalign{\smallskip} 0\\ \noalign{\smallskip} 0 \end{smallmatrix} \right), \left( \begin{smallmatrix} 0\\ \noalign{\smallskip} 3\\ \noalign{\smallskip} 0\\ \noalign{\smallskip} 0\\ \noalign{\smallskip} 0\\ \noalign{\smallskip} 0\\ \noalign{\smallskip} 0\\ \noalign{\smallskip} 0\\ \noalign{\smallskip} 0\\ \noalign{\smallskip} 0\\ \noalign{\smallskip} 0\\ \noalign{\smallskip} 0\\ \noalign{\smallskip} 0\\ \noalign{\smallskip} 0\\ \noalign{\smallskip} 25\\ \noalign{\smallskip} 75\\ \noalign{\smallskip} -75\\ \noalign{\smallskip} 0\\ \noalign{\smallskip} 0 \end{smallmatrix} \right), \left( \begin{smallmatrix} 0\\ \noalign{\smallskip} 0\\ \noalign{\smallskip} 3\\ \noalign{\smallskip} 0\\ \noalign{\smallskip} 0\\ \noalign{\smallskip} 0\\ \noalign{\smallskip} 0\\ \noalign{\smallskip} 0\\ \noalign{\smallskip} 0\\ \noalign{\smallskip} 0\\ \noalign{\smallskip} 0\\ \noalign{\smallskip} 0\\ \noalign{\smallskip} 0\\ \noalign{\smallskip} 0\\ \noalign{\smallskip} 400\\ \noalign{\smallskip} 1200\\ \noalign{\smallskip} 450\\ \noalign{\smallskip} 0\\ \noalign{\smallskip} 0 \end{smallmatrix} \right), \left( \begin{smallmatrix} 0\\ \noalign{\smallskip} 0\\ \noalign{\smallskip} 0\\ \noalign{\smallskip} 1\\ \noalign{\smallskip} 0\\ \noalign{\smallskip} 0\\ \noalign{\smallskip} 0\\ \noalign{\smallskip} 0\\ \noalign{\smallskip} 0\\ \noalign{\smallskip} 0\\ \noalign{\smallskip} 0\\ \noalign{\smallskip} 0\\ \noalign{\smallskip} 0\\ \noalign{\smallskip} 0\\ \noalign{\smallskip} 0\\ \noalign{\smallskip} 275\\ \noalign{\smallskip} 0\\ \noalign{\smallskip} 0\\ \noalign{\smallskip} 0 \end{smallmatrix} \right), \left( \begin{smallmatrix} 0\\ \noalign{\smallskip} 0\\ \noalign{\smallskip} 0\\ \noalign{\smallskip} 0\\ \noalign{\smallskip} 1\\ \noalign{\smallskip} 0\\ \noalign{\smallskip} 0\\ \noalign{\smallskip} 0\\ \noalign{\smallskip} 0\\ \noalign{\smallskip} 0\\ \noalign{\smallskip} 0\\ \noalign{\smallskip} 0\\ \noalign{\smallskip} 0\\ \noalign{\smallskip} 0\\ \noalign{\smallskip} 0\\ \noalign{\smallskip} -275\\ \noalign{\smallskip} 550\\ \noalign{\smallskip} 0\\ \noalign{\smallskip} 0 \end{smallmatrix} \right), \left( \begin{smallmatrix}0\\ \noalign{\smallskip} 0\\ \noalign{\smallskip} 0\\ \noalign{\smallskip} 0\\ \noalign{\smallskip} 0\\ \noalign{\smallskip} 1\\ \noalign{\smallskip} 0\\ \noalign{\smallskip} 0\\ \noalign{\smallskip} 0\\ \noalign{\smallskip} 0\\ \noalign{\smallskip} 0\\ \noalign{\smallskip} 0\\ \noalign{\smallskip} 0\\ \noalign{\smallskip} 0\\ \noalign{\smallskip} 0\\ \noalign{\smallskip} 825\\ \noalign{\smallskip} -2200\\ \noalign{\smallskip} 0\\ \noalign{\smallskip} 0 \end{smallmatrix} \right), \left( \begin{smallmatrix} 0\\ \noalign{\smallskip} 0\\ \noalign{\smallskip} 0\\ \noalign{\smallskip} 0\\ \noalign{\smallskip} 0\\ \noalign{\smallskip} 0\\ \noalign{\smallskip} 0\\ \noalign{\smallskip} 0\\ \noalign{\smallskip} 0\\ \noalign{\smallskip} 0\\ \noalign{\smallskip} 0\\ \noalign{\smallskip} 0\\ \noalign{\smallskip} 0\\ \noalign{\smallskip} 1\\ \noalign{\smallskip} 1\\ \noalign{\smallskip} 1\\ \noalign{\smallskip} 1\\ \noalign{\smallskip} 0\\ \noalign{\smallskip} 0 \end{smallmatrix} \right), \left( \begin{smallmatrix} 0\\ \noalign{\smallskip} 0\\ \noalign{\smallskip} 0\\ \noalign{\smallskip} 0\\ \noalign{\smallskip} 0\\ \noalign{\smallskip} 0\\ \noalign{\smallskip} 0\\ \noalign{\smallskip} 0\\ \noalign{\smallskip} 0\\ \noalign{\smallskip} 0\\ \noalign{\smallskip} 0\\ \noalign{\smallskip} 0\\ \noalign{\smallskip} 0\\ \noalign{\smallskip} 0\\ \noalign{\smallskip} 0\\ \noalign{\smallskip} 0\\ \noalign{\smallskip} 0\\ \noalign{\smallskip} 1\\ \noalign{\smallskip} 0 \end{smallmatrix} \right), \left( \begin{smallmatrix} 0\\ \noalign{\smallskip} 0\\ \noalign{\smallskip} 0\\ \noalign{\smallskip} 0\\ \noalign{\smallskip} 0\\ \noalign{\smallskip} 0\\ \noalign{\smallskip} 0\\ \noalign{\smallskip} 0\\ \noalign{\smallskip} 0\\ \noalign{\smallskip} 0\\ \noalign{\smallskip} 0\\ \noalign{\smallskip} 0\\ \noalign{\smallskip} 0\\ \noalign{\smallskip} 0\\ \noalign{\smallskip} 0\\ \noalign{\smallskip} 0\\ \noalign{\smallskip} 0\\ \noalign{\smallskip} 0\\ \noalign{\smallskip} 1 \end{smallmatrix} \right)\right\}.
\]
We can rewrite the matrix for $A_2^*$ in terms of these vectors:
\[
B:=\left(\begin{smallmatrix} 25&0&0&0&0&0&0&0&0
\\ \noalign{\medskip}0&3&0&0&0&0&0&0&0\\ \noalign{\medskip}0&0&3&0&0&0
&0&0&0\\ \noalign{\medskip}0&0&0&1&0&0&0&0&0\\ \noalign{\medskip}0&0&0
&0&1&0&0&0&0\\ \noalign{\medskip}0&0&0&0&0&1&0&0&0
\\ \noalign{\medskip}15300&75&-450&0&0&0&1&0&0\\ \noalign{\medskip}
5400&75&-450&0&0&0&0&1&0\\ \noalign{\medskip}0&0&0&0&0&0&0&0&1
\end {smallmatrix}\right)
\]
where the matrix $B$ is the unique homomorphism from $H^2(\Gamma/S^1)/\ker(A_2^*)$ to $H^2(\Gamma/S^1)/\ker(A_2^*)$ given by $B(g+\ker(A_2^*))=A_2^*g+\ker(A_2^*)$. Thus
 \[
\lim_{\longrightarrow}\left( H^2(\Gamma/S^1), A_2^*\right)\cong \lim_{\longrightarrow}\left( H^2(\Gamma/S^1)/\ker(A_2^*), B\right).
 \]

Using the generating set for $A_2^* H^2(\Gamma/S^1)$ above, we have $H^2(\Gamma/S^1)/\ker(A_2^*)\cong A_2^* H^2(\Gamma/S^1)\cong \mathbb{Z}^8 \oplus \mathbb{Z}_2$. Observe that the homomorphism $B$ is the identity on the 4th, 5th and 6th coordinates of $\mathbb{Z}^8$ and on $\mathbb{Z}_2$. Thus, $\check{H}^2(\Omega_0)$ is the direct sum of $\Z^3 \oplus \Z_2$ and $\displaystyle \lim_{\longrightarrow} ( \mathbb{Z}^5,  M)$
where
\[
M:=\left( \begin {smallmatrix} 25&0&0&0&0\\ \noalign{\medskip}0&3&0&0&0
\\ \noalign{\medskip}0&0&3&0&0\\ \noalign{\medskip}15300&75&-450&1&0
\\ \noalign{\medskip}5400&75&-450&0&1\end {smallmatrix} \right).
\]

The eigenvectors of $M$ are $(2,0,0,1275,450)$, $(0,6,1,0,0)$, and $(0,2,0,75,75)$ with eigenvalues $25$, $3$, and $3$ respectively. The quotient of  $\mathbb{Z}^5$ by the corresponding eigenspaces is isomorphic to $\mathbb{Z}^2$ and the quotient is generated by the classes of $(0,0,0,1,0)$ and $(0,0,0,0,1)$. We have the following short exact sequence
\begin{equation}\label{eq:seseq}
0 \to\mathbb{Z}^3 \to \mathbb{Z}^5 \to \mathbb{Z}^2 \to 0,
\end{equation}
where $\mathbb{Z}^3$ is the direct sum of the eigenspaces with eigenvalues $25, 3$ and $3$. Since $(0,0,0,1,0)$ and $(0,0,0,0,1)$ are eigenvectors of $M$ both with eigenvalue $1$, $M$ is the identity map on the quotient $\mathbb{Z}^2$. Taking direct limits of~\eqref{eq:seseq} gives the short exact sequence
\[
0 \to \mathbb{Z}[\frac{1}{25}] \oplus \mathbb{Z}[\frac{1}{3}]^2 \to \lim_{\longrightarrow} ( \mathbb{Z}^5,  M) \to \mathbb{Z}^2 \to 0.
\]
Since the quotient is free the sequence splits, and we have $\displaystyle \lim_{\longrightarrow} ( \mathbb{Z}^5,  M) = \mathbb{Z}[\frac{1}{25}] \oplus \mathbb{Z}[\frac{1}{3}]^2 \oplus \mathbb{Z}^2$. Hence 
\[
\check{H}^2(\Omega_0)=\lim_{\longrightarrow} ( \mathbb{Z}^5,  M) \oplus \mathbb{Z}^3 \oplus \mathbb{Z}_2 \cong \mathbb{Z}[\frac{1}{25}]\oplus \mathbb{Z}[\frac{1}{3}]^2 \oplus \mathbb{Z}^5 \oplus \Z_2. \qedhere
\]
\end{proof}

\section{Cohomology of $\Omega$}

Beginning with the integer cohomology of $\Omega_0$, the integer cohomology of $\Omega$ was computed in \cite[\S 7]{BDHS}. There are two steps in the computation, the first is to identify cone singularities and the second is to use a spectral sequence argument coming from the definition of the Gysin sequence. The second part of their computation was extremely difficult and we briefly outline the steps in the calculation. Using the results of \cite[\S 7]{BDHS} we give the integer cohomology of the Pinwheel tiling.

We begin by identifying the cone singularities. In the Pinwheel tiling, cone singularities correspond to Pinwheel tilings with $180^\circ$ symmetry. Rotational symmetry is preserved under substitution so we can identify cone singularities by finding periodic points under the substitution of $180^\circ$ symmetric patches of collared tiles. From Appendix \ref{collared} we find the following patches with $180^\circ$ symmetry and vertices at the centre of rotation that are periodic under the substitution.
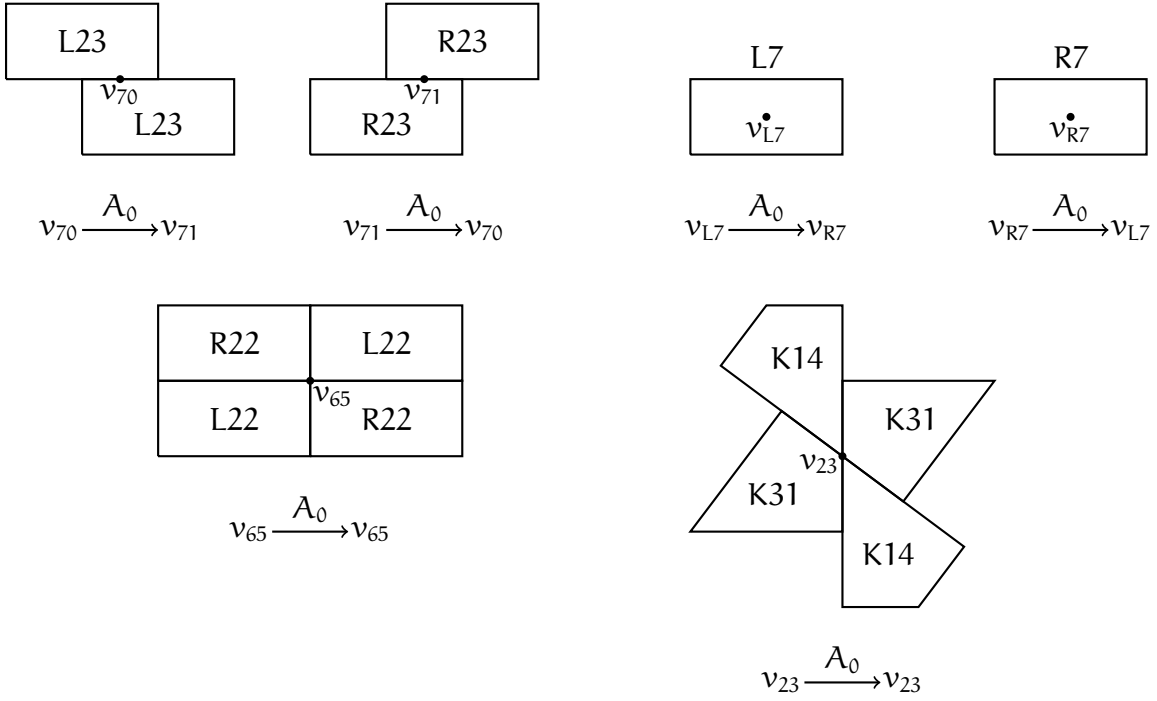
\begin{center}
\begin{tikzpicture}[scale=1.0]

\begin{scope}[thick,xshift=0cm,yshift=0cm]
    \draw (0,0)--(2,0)--(2,1)--(0,1)--(0,0);
    \draw (-1,1)--(1,1)--(1,2)--(-1,2)--(-1,1);
    \pgfpathcircle{\pgfpoint{0.5cm}{1cm}} {1.5pt};
    \pgfusepath{fill}
    \node at (0,1.5) {$L23$};
    \node at (1,0.4) {$L23$};
    \node at (0.5,0.8) {$v_{70}$};
    \draw[->,thick] (0,-1)--(1,-1);
    \node at (0.5,-0.7) {$A_0$};
    \node at (1.3,-1) {$v_{71}$};
    \node at (-0.3,-1) {$v_{70}$};
\end{scope}

\begin{scope}[thick,xshift=4cm,yshift=0cm]
    \draw (0,1)--(2,1)--(2,2)--(0,2)--(0,1);
    \draw (-1,0)--(1,0)--(1,1)--(-1,1)--(-1,0);
    \pgfpathcircle{\pgfpoint{0.5cm}{1cm}} {1.5pt};
    \pgfusepath{fill}
    \node at (1,1.5) {$R23$};
    \node at (0,0.4) {$R23$};
    \node at (0.5,0.8) {$v_{71}$};
    \draw[->,thick] (0,-1)--(1,-1);
    \node at (0.5,-0.7) {$A_0$};
    \node at (1.3,-1) {$v_{70}$};
    \node at (-0.3,-1) {$v_{71}$};
\end{scope}

\begin{scope}[thick,xshift=8cm,yshift=0cm]
    \draw (0,0)--(2,0)--(2,1)--(0,1)--(0,0);
    \pgfpathcircle{\pgfpoint{1cm}{0.5cm}} {1.5pt};
    \pgfusepath{fill}
    \node at (1,1.3) {$L7$};
    \node at (1,0.3) {$v_{L7}$};
    \draw[->,thick] (0.5,-1)--(1.5,-1);
    \node at (1,-0.7) {$A_0$};
    \node at (1.8,-1) {$v_{R7}$};
    \node at (0.2,-1) {$v_{L7}$};
\end{scope}

\begin{scope}[thick,xshift=12cm,yshift=0cm]
    \draw (0,0)--(2,0)--(2,1)--(0,1)--(0,0);
    \pgfpathcircle{\pgfpoint{1cm}{0.5cm}} {1.5pt};
    \pgfusepath{fill}
    \node at (1,1.3) {$R7$};
    \node at (1,0.3) {$v_{R7}$};
    \draw[->,thick] (0.5,-1)--(1.5,-1);
    \node at (1,-0.7) {$A_0$};
    \node at (1.8,-1) {$v_{L7}$};
    \node at (0.2,-1) {$v_{R7}$};
\end{scope}

\begin{scope}[thick,xshift=3cm,yshift=-4cm]
    \draw (0,0)--(2,0)--(2,1)--(0,1)--(0,0);
    \draw (0,1)--(2,1)--(2,2)--(0,2)--(0,1);
    \draw (-2,1)--(0,1)--(0,2)--(-2,2)--(-2,1);
    \draw (-2,0)--(0,0)--(0,1)--(-2,1)--(-2,0);
    \pgfpathcircle{\pgfpoint{0cm}{1cm}} {1.5pt};
    \pgfusepath{fill}
    \node at (-1,1.5) {$R22$};
    \node at (-1,0.5) {$L22$};
    \node at (1,1.5) {$L22$};
    \node at (1,0.5) {$R22$};
    \node at (0.3,0.8) {$v_{65}$};
    \draw[->,thick] (-0.5,-1)--(0.5,-1);
    \node at (0,-0.7) {$A_0$};
    \node at (0.8,-1) {$v_{65}$};
    \node at (-0.8,-1) {$v_{65}$};
\end{scope}

\begin{scope}[thick,xshift=10cm,yshift=-4cm]
    \draw (0,0)--(0,2)--(-1,2)--(-1.6,1.2)--(0,0);
    \draw (0,0)--(0,-2)--(1,-2)--(1.6,-1.2)--(0,0);
    \draw (0,0)--(0,1)--(2,1)--(0.8,-0.6)--(0,0);
    \draw (0,0)--(0,-1)--(-2,-1)--(-0.8,0.6)--(0,0);
    \pgfpathcircle{\pgfpoint{0cm}{0cm}} {1.5pt};
    \pgfusepath{fill}
    \node at (-0.6,1.3) {$K14$};
    \node at (0.6,-1.3) {$K14$};
    \node at (-0.9,-0.5) {$K31$};
    \node at (0.9,0.5) {$K31$};
    \node at (-0.3,-0.1) {$v_{23}$};
    \draw[->,thick] (-0.5,-3)--(0.5,-3);
    \node at (0,-2.7) {$A_0$};
    \node at (0.8,-3) {$v_{23}$};
    \node at (-0.8,-3) {$v_{23}$};
\end{scope}
\end{tikzpicture}
\end{center}
Placing the vertex of rotation at the origin of $\R^2$ and substituting an infinite number of times leads to a tiling of the plane with $180^\circ$ symmetry \cite[Theorem 1.4]{Lorenzo.book}. Therefore, there are 6 cone singularities. We note that there are four additional patches with $180^\circ$ symmetry (using tiles $L3$, $L26$, $R3$, and $R26$), however the vertex at the centre of rotation is not periodic under the substitution so they do not contribute to the number of cone singularities. 

We now briefly outline the spectral sequence argument in \cite[\S 7]{BDHS} used to compute $\check{H}^*(\Omega)$. The Gysin sequence \cite[p.177]{BT} is used to compute the cohomology of an $S^1$-fibre bundle over the base space from the cohomology of the base space. Unfortunately, due to the presence of cone singularities, the map $\Omega \to \Omega_0$ is not quite a fibration. However, the spectral sequence used to prove the Gysin sequence can be adjusted to account for these cone singularities. In \cite{BDHS}, the authors show that the generators of the cone singularities account for 5 linearly independent torsion elements contributing a $\Z_2^5$ term to $\check{H}^2(\Omega)$. The relevant results of \cite{BDHS} are summarised in the following theorem.

\begin{thm}[Barge, Diamond, Hunton, and Sadun {\cite[\S7 and Theorem 12]{BDHS}}]
Suppose $\Omega$ is the continuous hull of the Pinwheel tiling and $\Omega_0 = \Omega/S^1$. Given $\check{H}^*(\Omega_0)$ and the cone singularities corresponding to tilings with $180^\circ$-rotational symmetry. The integer cohomology groups are given by
\begin{align*}
\check{H}^0(\Omega) &= \check{H}^0(\Omega_0) = \Z \\
\check{H}^1(\Omega) &= \check{H}^0(\Omega_0) \oplus \check{H}^1(\Omega_0) = \Z \oplus \Z \\
\check{H}^2(\Omega) &= E^\infty_{1,1} \oplus E^\infty_{2,0} = \mathbb{Z}[\frac{1}{25}]\oplus\mathbb{Z}[\frac{1}{3}]^2\oplus \mathbb{Z}^6 \oplus \mathbb{Z}_2^5 \\
\check{H}^3(\Omega) &= \check{H}^2(\Omega_0) = \mathbb{Z}[\frac{1}{25}]\oplus\mathbb{Z}[\frac{1}{3}]^2\oplus \mathbb{Z}^5 \oplus \mathbb{Z}_2
\end{align*}
where the terms $E^\infty_{1,1}=\Z \oplus \Z_2^5$ and $E^\infty_{2,0}= \mathbb{Z}[\frac{1}{25}]\oplus\mathbb{Z}[\frac{1}{3}]^2\oplus \mathbb{Z}^5$ come from the spectral sequence $E^\infty$ which adjusts for the cone singularities and the torsion in $\check{H}^2(\Omega_0)$.
\end{thm}

\newpage

\begin{appendix}

\section{The collared tiles}
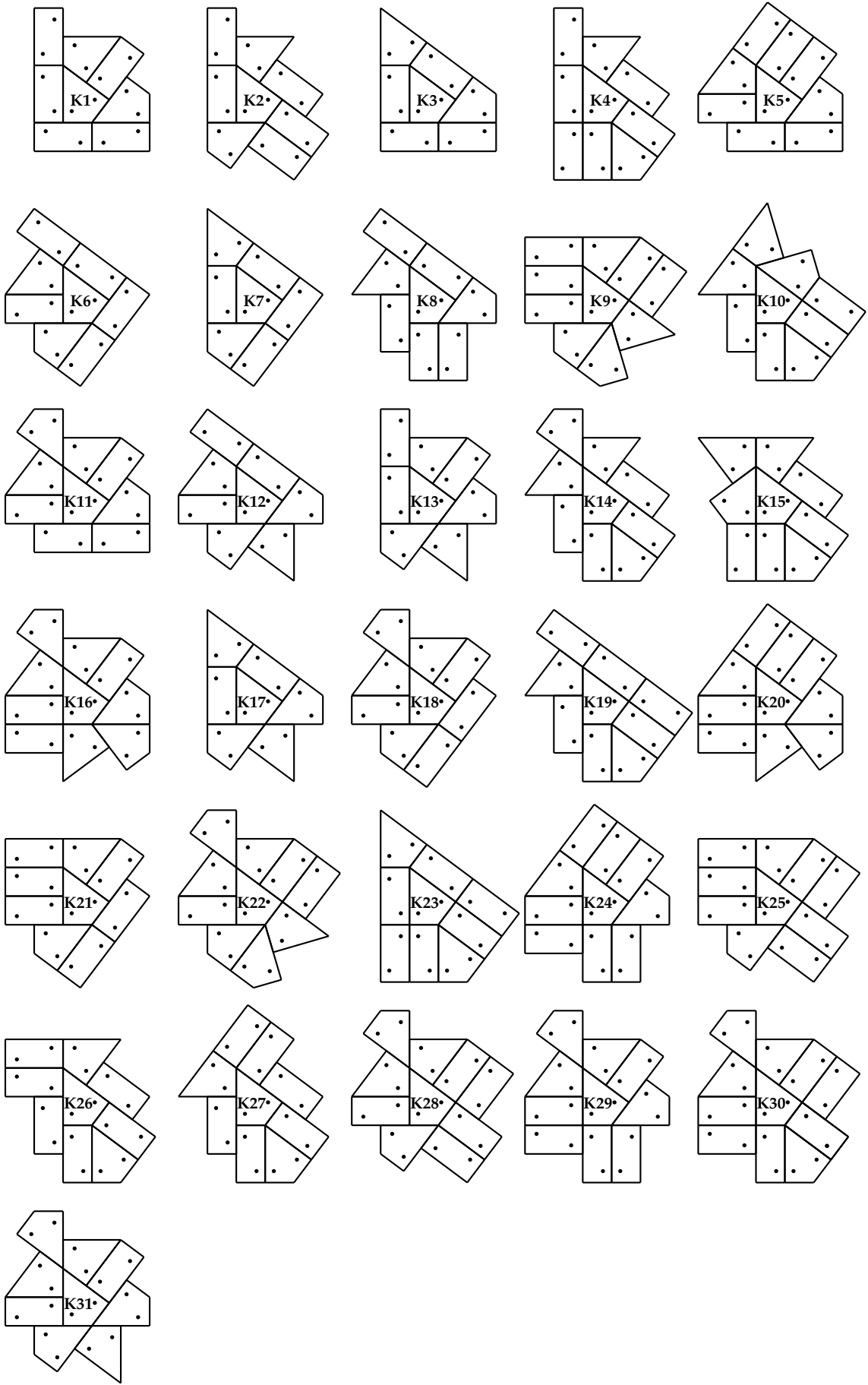
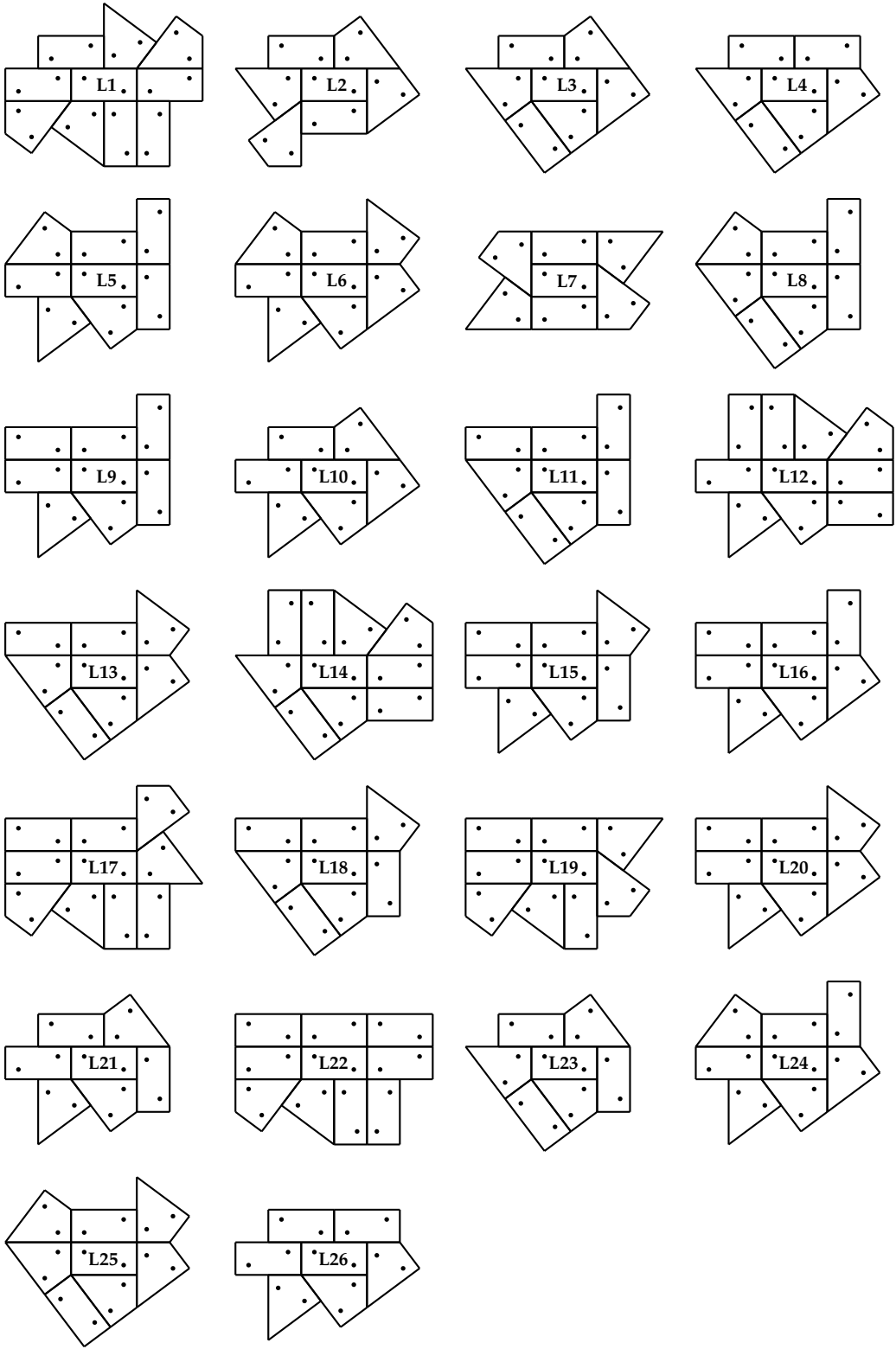
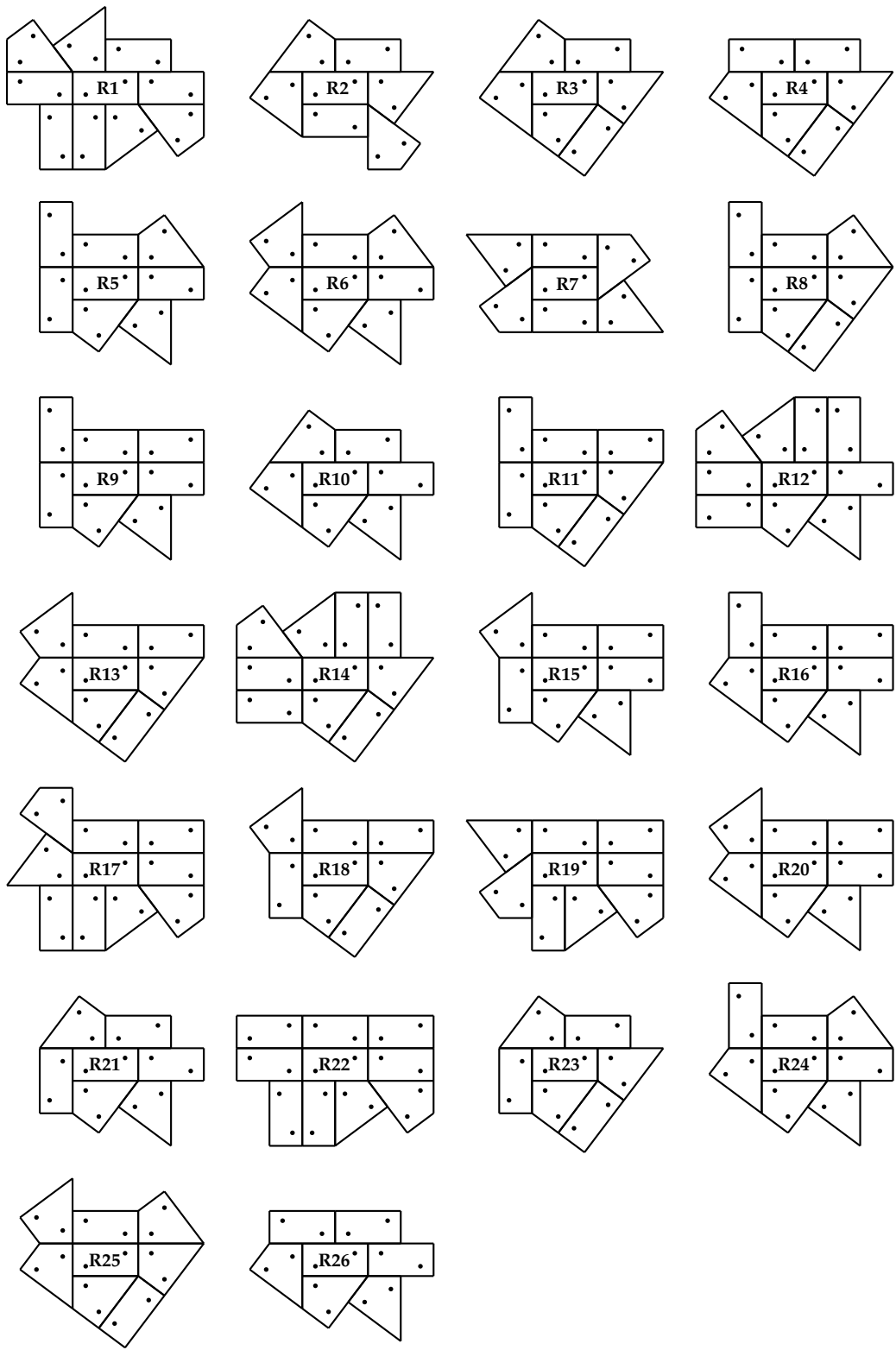
\label{collared}

\begin{center}

\end{center}

\section{The labelled tiles}\label{app_tiles}

\begin{center}
\scalebox{0.8}{
%
}
\end{center}


\section{The vectors generating $H^2(\Gamma/S^1)$}\label{matrices}

\begin{center}
\tiny
\scalebox{0.60}{
$\left \{ \left(%
 \right) \right \}.$ }
 \end{center}
 

\end{appendix}

\end{document}